\documentclass[12pt]{amsart}
\usepackage{amsmath,amsthm,amssymb,amsfonts,graphicx,color,multirow,microtype,array,bigstrut,multicol}
\usepackage[left=2.25cm,right=2.25cm,bottom=2.25cm,top=3.25cm]{geometry}
\usepackage{tikz}
\usepackage{import}
\usepackage{hyperref}
\usepackage{thmtools}
\usepackage{thm-restate}
\numberwithin{equation}{section}
\numberwithin{figure}{section}
\theoremstyle{plain}

\newtheorem{theorem}{Theorem}[section]
\newtheorem{lemma}[theorem]{Lemma}
\newtheorem{proposition}[theorem]{Proposition}
\newtheorem{corollary}[theorem]{Corollary}
\newtheorem{identity}[theorem]{Identity}

\theoremstyle{definition}
\newtheorem{definition}[theorem]{Definition}

\theoremstyle{remark}
\newtheorem{remark}[theorem]{Remark}
\newtheorem{example}[theorem]{Example}

\begin{document}

\title[Graded Expectations]{Graded Expectations:  Betti numbers and anti-lecture hall compositions of random threshold graphs \\  \bigskip A tale of probabilistic commutative algebra}

\author[A. Engstr\"om]{Alexander Engstr\"om}
\address{Department of Mathematics, Aalto University, Helsinki, Finland}
\email{alexander.engstrom@aalto.fi}

\author[C. Go]{Christian Go}
\address{Department of Mathematics, National University of Singapore, Singapore}
\email{christian.go@nus.edu.sg}

\author[M.T. Stamps]{Matthew T. Stamps}
\address{Division of Science, Yale-NUS College, Singapore}
\email{matt.stamps@yale-nus.edu.sg}

\begin{abstract}
This paper examines the one-to-one-to-one correspondence between threshold graphs, Betti numbers of quotients of polynomial rings by $2$-linear ideals, and anti-lecture hall compositions.  In particular, we establish new explicit combinatorial mappings between each of these classes of objects and calculate the expected values of the Betti numbers and anti-lecture hall composition corresponding to a random threshold graph.
\end{abstract}

\maketitle

\section{Introduction}

A fundamental task in commutative algebra, originating with Hilbert, is to characterize the graded Betti numbers of finitely-generated modules over a polynomial ring.  This largely eluded mathematicians until the emergence of Boij-S\"oderberg theory, which established that Betti diagrams can be decomposed as rational combinations of a prescribed set of tables.  In addition to overcoming the challenge of characterization, Boij-S\"oderberg theory has a natural combinatorial flavor that has strengthened the role of combinatorial commutative algebra (see the papers of Cook II \cite{Cook}, Erman and Sam \cite{ES17}, Herzog, Sharifan, and Varbaro \cite{HSV14}, or Nagel and Sturgeon \cite{NS13} for specific examples).  In the paper \cite{ES13}, Engstr\"om and Stamps showed that the Betti numbers of quotients of polynomial rings by the well-studied class of $2$-linear ideals are in one-to-one correspondence with a well-studied class of graphs called threshold graphs.  The major upshot of this correspondence is that one can not only determine if a table of integers is the Betti diagram associated to one of those ideals, but also quickly find an explicit example when the answer is positive.   

This paper aims to enrich the one-to-one-to-one correspondence between threshold graphs, Betti numbers of quotients by $2$-linear ideals, and anti-lecture hall compositions by establishing new explicit combinatorial mappings between each pair of classes of these objects.  We also consider a probabilistic approach to commutative algebra in which we try to understand what properties one can expect to observe in the Betti diagram of an ideal sampled from a given distribution on the set of possible Betti diagrams.  For the case of $2$-linear ideals, we are able to use the correspondence in \cite{ES13} and a random model for threshold graphs to calculate expected values of the Betti numbers and anti-lecture hall composition for a $1$-parameter family of distributions, including the uniform distribution.

We have organized the paper as follows:  Section 2 reviews the main objects of interest and describes the one-to-one-to-one correspondence between them.  Section 3 then presents new explicit combinatorial mappings between the main objects. Finally, Section 4 contains closed formulas for the expected Betti numbers and anti-lecture hall composition associated to a random threshold graph. We have attempted to use standard notation wherever possible to make it easier for the more knowledgeable reader to skip ahead to later sections.

\section{Preliminaries}

We begin by reviewing the central objects of this paper, revisiting the main definitions and the relevant results regarding threshold graphs, Betti numbers of quotients of polynomial rings, and anti-lecture hall compositions.

\subsection{Threshold Graphs} 

The first central class of objects we consider in this paper are threshold graphs.  A \emph{graph} is a pair $G = (V,E)$ where $V$ is a finite set and $E$ is a subset of pairs of elements in $V$.  The set $V = V(G)$ is called the \emph{vertex set} of $G$ and the set $E = E(G)$ is called the \emph{edge set} of $G$. For a comprehensive introduction to graph theory, we recommend the textbooks by Bollob\'as~\cite{Bollobas} and Diestel~\cite{Diestel}.


\begin{definition}
A graph $G$ is called \emph{threshold} if there exists a function $\omega : V(G) \to \mathbb{R}$ and a real number $t \in \mathbb{R}$ such that $uv \in E(G)$ if and only if $\omega(u) + \omega(v) \geq t$.  
\end{definition}

\begin{example}
The following graph is threshold by taking $t = 3$, $\omega(0) = \omega(1) = \omega(2) = 1$, $\omega(3) = 2$, $\omega(4) = 0$, and $\omega(5) = 3$.  
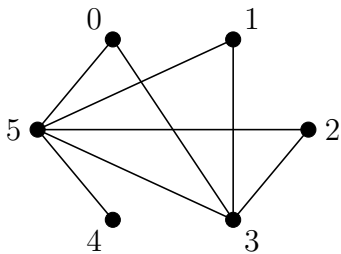
\begin{figure}
    \centering
\begin{tikzpicture}
\draw [fill] (0,2.4) circle [radius=0.1];
\node [above left] at (0,2.4) {$0$};
\draw [fill] (1.6,2.4) circle [radius=0.1];
\node [above right] at (1.6,2.4) {$1$};
\draw [fill] (2.6,1.2) circle [radius=0.1];
\node [right] at (2.6,1.2) {$ \, 2$};
\draw [fill] (1.6,0) circle [radius=0.1];
\node [below right] at (1.6,0) {$3$};
\draw [fill] (0,0) circle [radius=0.1];
\node [below left] at (0,0) {$4$};
\draw [fill] (-1,1.2) circle [radius=0.1];
\node [left] at (-1,1.2) {$5 \, $};
\draw [semithick] (1.6,0) -- (0,2.4);
\draw [semithick] (1.6,0) -- (1.6,2.4);
\draw [semithick] (1.6,0) -- (2.6,1.2);
\draw [semithick] (-1,1.2) -- (0,2.4);
\draw [semithick] (-1,1.2) -- (1.6,2.4);
\draw [semithick] (-1,1.2) -- (2.6,1.2);
\draw [semithick] (-1,1.2) -- (1.6,0);
\draw [semithick] (-1,1.2) -- (0,0);
\end{tikzpicture}
    \caption{Threshold graph on six vertices.}
    \label{fig:thresholdgraphexample}
\end{figure}
\end{example}

Threshold graphs were introduced by Chvat\'al and Hammer in \cite{CH77} and \cite{CH73} and have been studied extensively from structural graph theory, (see Mahadev and Peled's book \cite{MP95}) to complex networks (see the paper of Hagberg, Swart, and Schult \cite{HSS06}). They are particularly convenient to work with computationally---Heggernes and Kratsch showed in \cite{HK07}, for instance, that one can determine whether or not a graph is threshold in linear time. Furthermore, they have several equivalent characterizations which are summarized in Theorem 1.2.4 of \cite{MP95}.  The most relevant characterization for the paper at hand is that threshold graphs can be constructed from a single vertex by applying a sequence of two operations. 

In a graph $G$ with vertex set $V(G) = \{0,1,\ldots,m\}$, a vertex $v$ is called \emph{dominating} if $uv \in E(G)$ for every $u < v$ in $V(G)$ and \emph{isolated} if $uv \notin E(G)$ for every $u < v$ in $V(G)$. 

\begin{proposition}[Theorem 1.2.4, \cite{MP95}]\label{prop:char}
A graph $G$ on $m+1$ vertices is threshold if and only if its vertices can be labeled $0$ through $m$ such that every vertex in $\{1,\ldots,m\}$ is isolated or dominating.  
\end{proposition}

For $n \in \mathbb{N}$ and $\sigma \subseteq \{1,\ldots,n\}$, let $T(n,\sigma)$ denote the threshold graph with vertex set $\{0,1,\ldots,n\}$ whose isolated vertices are precisely the elements in $\sigma$.  In this notation, the graph in Figure~\ref{fig:thresholdgraphexample} is $T(5,\{1,2,4\})$.  The characterization in Proposition~\ref{prop:char} has the additional property that a threshold graph is completely determined by its isolated (or, equivalently, its dominating) vertices:

\begin{theorem}[Theorem 3.2.2, \cite{MP95}] A pair of threshold graphs $T(n,\sigma)$ and $T(m,\tau)$ are isomorphic if and only if $n = m$ and $\sigma = \tau$.  
\end{theorem}

It is thus straightforward to enumerate the threshold graphs on a given set of vertices.

\begin{corollary}
There are $2^n$ threshold graphs on $n+1$ vertices up to isomorphism. 
\end{corollary}

Threshold graphs belong to a larger class of graphs called \emph{chordal} graphs.  To define chordal graphs, we need the notion of an \emph{induced cycle}:  given a graph $G$ and a subset $W \subseteq V(G)$, the \emph{induced subgraph} of $G$ on $W$ is the graph $G[W]$ whose vertex set is $W$ and whose edge set is the set of all edges in $G$ with both endpoints in $W$.  A \emph{cycle} of length $m$ is a collection of edges of the form $$v_1v_2, v_2v_3, \ldots, v_{m-1}v_m, \text{ and } v_1v_m.$$ 

\begin{definition}
A graph $G$ is \emph{chordal} if it has no induced cycles of length more than $3$.
\end{definition}

It is well known that threshold graphs are chordal, but not all chordal graphs are threshold. Consider, for instance, the graphs in the following example:

\begin{example}
\begin{figure}\label{fig:example_graphs}
    \centering
\begin{tikzpicture}
\draw [fill] (0,1.7) circle [radius=0.1];
\node [above left] at (0,1.7) {$0$};
\draw [fill] (1.7,1.7) circle [radius=0.1];
\node [above right] at (1.7,1.7) {$1$};
\draw [fill] (1.7,0) circle [radius=0.1];
\node [below right] at (1.7,0) {$2$};
\draw [fill] (0,0) circle [radius=0.1];
\node [below left] at (0,0) {$3$};
\draw [semithick] (1.7,0) -- (0,0);
\draw [semithick] (0,1.7) -- (0,0);
\draw [semithick] (1.7,1.7) -- (0,0);
\end{tikzpicture}
\qquad
\quad
\begin{tikzpicture}
\draw [fill] (0,1.7) circle [radius=0.1];
\node [above left] at (0,1.7) {$0$};
\draw [fill] (1.7,1.7) circle [radius=0.1];
\node [above right] at (1.7,1.7) {$1$};
\draw [fill] (1.7,0) circle [radius=0.1];
\node [below right] at (1.7,0) {$2$};
\draw [fill] (0,0) circle [radius=0.1];
\node [below left] at (0,0) {$3$};
\draw [semithick] (0,0) -- (0,1.7);
\draw [semithick] (0,1.7) -- (1.7,1.7);
\draw [semithick] (1.7,1.7) -- (1.7,0);
\end{tikzpicture}
\qquad
\quad
\begin{tikzpicture}
\draw [fill] (0,1.7) circle [radius=0.1];
\node [above left] at (0,1.7) {$0$};
\draw [fill] (1.7,1.7) circle [radius=0.1];
\node [above right] at (1.7,1.7) {$1$};
\draw [fill] (1.7,0) circle [radius=0.1];
\node [below right] at (1.7,0) {$2$};
\draw [fill] (0,0) circle [radius=0.1];
\node [below left] at (0,0) {$3$};
\draw [semithick] (0,0) -- (0,1.7);
\draw [semithick] (0,1.7) -- (1.7,1.7);
\draw [semithick] (1.7,1.7) -- (1.7,0);
\draw [semithick] (1.7,0) -- (0,0);
\end{tikzpicture}
    \caption{Examples and non-examples of threshold and chordal graphs}
\end{figure}
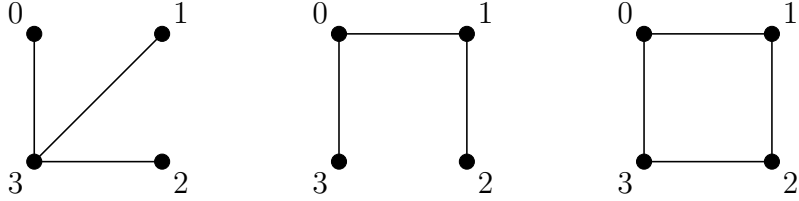
The graph on the left in Figure~\ref{fig:example_graphs} is the threshold graph $T(3,\{1,2\})$, which means it is necessarily chordal.  The graph in the middle is chordal since it has no cycles, but is not threshold; otherwise there would exist a set of weights and a threshold $t$ such that $\sum\limits_{v = 0}^3 \omega(v) \geq 2t$ (since $03$ and $12$ are edges) and $\sum\limits_{v=0}^3 \omega(v) < 2t$ (since $02$ and $13$ are non-edges), which is impossible. The graph on the right is not chordal since it contains an induced cycle of length $4$, which means it also cannot be threshold.
\end{example}

\subsection{Betti Numbers}

The second central class of objects we consider in this paper, Betti numbers of finitely-generated graded modules over polynomial rings, are fundamental invariants in commutative algebra. For a comprehensive introduction to the topic, we recommend the textbooks of Eisenbud \cite{Eisenbud} and Miller and Sturmfels \cite{MS}.

Let $\Bbbk$ be a field, let $S = \Bbbk\left[x_1,\ldots,x_n\right]$, let $M$ be a finitely-generated graded $S$-module, and let $M(d)$ denote the \emph{twisting} of $M$ by $d$, that is, let $M(d)$ be the module whose $i$-th graded piece is $M_{i+d}$.

\begin{definition}
A  graded \emph{free resolution} of $M$ is an exact complex of the form $$\mathcal{F}_{\bullet}:0\longleftarrow M\overset{\phi_0}{\longleftarrow} F_{0}\overset{\phi_{1}}{\longleftarrow}F_{1}\longleftarrow\cdots\longleftarrow F_{\ell-1}\overset{\phi_{\ell}}{\longleftarrow}F_{\ell}\longleftarrow0,$$ where each $F_i$ is a graded free $S$-module and each homomorphism $\phi_i$ is degree-preserving. \end{definition}

The \emph{length} of $\mathcal{F}_{\bullet}$ is the largest value $\ell$ such that $F_{\ell} \neq 0$, and by the Hilbert Syzygy Theorem, every $S$-module has a free resolution of length at most $n$ (see Miller and Sturmfels' book \cite{MS} for more details).  A free resolution of $M$ is \emph{minimal} if each $F_{i}$ has the minimal number of generators.  It is known that every finitely-generated graded $S$-module has a minimal free resolution that is unique up to isomorphism (see \cite{Eisenbud}). The \emph{projective dimension} of a module is the length of its minimal resolution. 

\begin{definition}
Let $\mathcal{F}_{\bullet}$ be a minimal free resolution of $M$ with $$F_{i}=\bigoplus_{j\in\mathbb{Z}}S\left(-j\right)^{\beta_{i,j}}$$ where $S\left(-j\right)$ is the twisting of $S$ by $-j$. That is, let $\mathcal{F}_{\bullet}$ be of the form
$$0\longleftarrow M\longleftarrow\bigoplus_{j\in\mathbb{Z}}S\left(-j\right)^{\beta_{0,j}}\overset{\phi_{1}}{\longleftarrow}\cdots\overset{\phi_{\ell}}{\longleftarrow}\bigoplus_{j\in\mathbb{Z}}S\left(-j\right)^{\beta_{\ell,j}}\longleftarrow0.$$  The exponent $\beta_{i,j}$ is called the \emph{$i^{\mathrm{th}}$ graded Betti number of degree $j$} for $M$.
 \end{definition}

\begin{example}\label{ex:linear}
Let $S=\Bbbk\left[x_{1},x_{2},x_{3}\right]$ and $I=\left<x_{1}x_{2},x_{1}x_{3},x_{2}x_{3}\right>\subseteq S$. A minimal graded free resolution of $S/I$ is
$$0\longleftarrow S/I\overset{\pi}{\longleftarrow}S\left(0\right)^{1}\xleftarrow{{\scriptscriptstyle {\left[\begin{smallmatrix}x_{1}x_{2} & x_{1}x_{3} & x_{2}x_{3}\end{smallmatrix}\right]}}}S\left(-2\right)^{3}\xleftarrow{{\scriptscriptstyle {\left[\begin{smallmatrix}-x_{3} & 0\\
x_{2} & -x_{2}\\
0 & x_{1}
\end{smallmatrix}\right]}}}S\left(-3\right)^{2}\longleftarrow0$$
and its nonzero Betti numbers are $\beta_{0,0} = 1$, $\beta_{1,2} = 3$, and $\beta_{2,3}=2$.
\end{example}

\begin{example}\label{ex:nonlinear}
Let $S=\Bbbk\left[x_{1},x_{2}\right]$ and $I=\left<x_{1}^{2},x_{1}x_{2},x_{2}^{3}\right>\subseteq S$. A minimal free resolution of $S/I$ is
$$0\longleftarrow S/I\overset{\pi}{\longleftarrow}S\left(0\right)^{1}\xleftarrow{{\scriptscriptstyle {\left[\begin{smallmatrix}x_{1}^{2} & x_{1}x_{2} & x_{2}^{3}\end{smallmatrix}\right]}}}S\left(-2\right)^{2}\oplus S\left(-3\right)^{1}\xleftarrow{{\scriptscriptstyle {\left[\begin{smallmatrix}x_{2} & 0\\
-x_{1} & x_{2}^{2}\\
0 & -x_{1}
\end{smallmatrix}\right]}}}S\left(-3\right)^{1}\oplus S\left(-4\right)^{1}\leftarrow0$$
and its nonzero Betti numbers are $\beta_{0,0}=1$, $\beta_{1,2}=2$, $\beta_{1,3}=1$, $\beta_{2,3}=1,$ and $\beta_{2,4}=1$.
\end{example}

An ideal $I$ in $S$ is \emph{$d$-linear} if $\beta_{i,j}\left(S/I\right)=0$ whenever $j-i\neq d-1$ except for the case where $i = j = 0$ since $\beta_{0,0}(S/I)$ is necessarily equal to $1$.  The ideal in Example~\ref{ex:linear} is $2$-linear whereas the ideal in Example~\ref{ex:nonlinear} is not.  From here on, we shall restrict our attention to Betti numbers of quotients by $2$-linear ideals.  Since there is at most one nonzero $\beta_{i,j}$ for each $0 \leq i \leq n$ in a $2$-linear ideal and $\beta_{0,1} = 0$ for every quotient $S/I$, we shall simplify our notation for the Betti numbers to  $$\beta(S/I) = (\beta_1,\ldots,\beta_n) = (\beta_{1,2}, \ldots, \beta_{n,n+1})$$ which we will call the \emph{Betti sequence} of $S/I$.

\subsection{Anti-Lecture Hall Compositions}

The third and final central class of objects we consider in this paper are sequences of numbers called anti-lecture hall compositions. 

\begin{definition}
An \emph{anti-lecture hall composition} of length $n$ bounded above by $t$ is an integer sequence $\lambda=\left(\lambda_{1},\dots,\lambda_{n}\right)$ that satisfies $$t\geq\frac{\lambda_{1}}{1}\geq\frac{\lambda_{2}}{2}\geq\cdots\geq\frac{\lambda_{n}}{n}\geq0.$$
\end{definition}

\begin{example}
The sequence $(1,2,2,1,0)$ is an anti-lecture hall composition of length $5$ bounded above by $1$ since $$1 \geq \frac{1}{1} \geq \frac{2}{2} \geq \frac{2}{3} \geq \frac{1}{4} \geq \frac{0}{5} \geq 0.$$
\end{example}

Anti-lecture hall compositions arose naturally from the study of the more intuitively named \emph{lecture hall partitions} that were introduced by Bousquet-M\'elou and Eriksson in \cite{BE1} and \cite{BE2}.  These sequences have become fundamental objects in algebraic and enumerative combinatorics (see Savage's paper \cite{S16} for a comprehensive survey).  We do not require a substantial portion of the theory of anti-lecture hall compositions in this paper, but we will make use of the following result of Corteel, Lee, and Savage in \cite{CLS}.

\begin{theorem}[Corollary 4, \cite{CLS}]
There are $(t+1)^n$ anti-lecture hall compositions of length $n$ bounded above by $t$.  
\end{theorem}

In particular, there are $2^n$ anti-lecture hall compositions of length $n$ bounded above by $1$.  As we alluded to in the introduction, this set has the same cardinality as the set of threshold graphs on $n+1$ vertices.  We will discuss a specific bijection between these two sets in the next section.

\subsection{The Correspondence}

Let $\mathcal{T}_n$, $\mathcal{B}_n$, and $\mathcal{A}_n$ denote the sets of threshold graphs on $n+1$ vertices, Betti sequences of quotients of $\Bbbk[x_0,\ldots,x_n]$ by $2$-linear ideals, and anti-lecture hall compositions of length $n$ bounded above by $1$, respectively.  Engstr\"om and Stamps discovered a one-to-one-to-one correspondence between these three sets in \cite{ES13}.

\begin{theorem}[Proposition 4.11, \cite{ES13}]\label{thm:correspondence}
The sets $\mathcal{T}_n$, $\mathcal{B}_n$, and $\mathcal{A}_n$ are in one-to-one-to-one correspondence for every $n \in \mathbb{N}$.
\end{theorem}

The correspondence between $\mathcal{T}_n$ and $\mathcal{B}_n$ follows from a standard construction in combinatorial commutative algebra:  To every graph $G$ with vertex set $V(G) = \{0,1,\ldots,n\}$, we define the \emph{coedge ideal} of $G$ in $R = \Bbbk[x_0,x_1, \ldots, x_n]$ by $$I_c(G) = \langle x_ux_v \ | \ uv \notin E(G)\rangle.$$  Fr\"oberg \cite{F90} proved that the quotient ring $\Bbbk[G] = R/I_{c}\left(G\right)$ is $2$-linear if and only if $G$ is chordal.  Thus, since threshold graphs are chordal, $G \mapsto \beta(\Bbbk[G])$ gives a natural mapping from $\mathcal{T}_n$ to $\mathcal{B}_n$.  Engstr\"om and Stamps gave an explicit proof that this mapping is injective in \cite{ES13}, though it appears that experts have been aware of this fact for some time.  

The correspondence between $\mathcal{B}_n$ and $\mathcal{A}_n$ follows from the Boij-S\"oderberg theorems, which were conjectured by Boij and S\"oderberg in \cite{BS1}, proven by Eisenbud and Schreyer in \cite{ES09} and extended by Boij and S\"oderberg in \cite{BS2}.  The theorems state that the Betti diagrams of a finitely-generated graded $S$-module can be uniquely decomposed into a rational combination of a prescribed set of Betti diagrams (see Fl{\o}ystad's survey \cite{F12} for a comprehensive introduction).  The set $\mathcal{A}_n$ arises as the image of an invertible linear transformation on the Boij-S\"oderberg coefficients from $\mathcal{B}_n$ (this is shown explicitly in \cite{ES13}), thus yielding an injective mapping from $\mathcal{B}_n$ to $\mathcal{A}_n$.  

Theorem~\ref{thm:correspondence} is then a straightforward consequence of the following two facts: first, that $|\mathcal{T}_n| = |\mathcal{A}_n|$, and second, that the composition $\mathcal{T}_n \to \mathcal{B}_n \to \mathcal{A}_n$ is injective.  As an explicit example, the correspondence for $n = 3$ has been listed in Table~\ref{tab:corr}.

\begin{table}
\setlength{\tabcolsep}{12pt}
\renewcommand{\arraystretch}{1.25}
\begin{center}
\begin{tabular}{|m{1cm}|c|c|c|}
\hline 
\multicolumn{1}{|c}{$\mathcal{T}_3$} & \multicolumn{1}{|c}{$2^{[3]}$} & \multicolumn{1}{|c}{$\mathcal{B}_3$} & \multicolumn{1}{|c|}{$\mathcal{A}_3$} \\
\hline 
\hline 
\vspace{0pt}
\begin{tikzpicture}[scale=0.8]
\draw [fill] (0,1) circle [radius=0.08];
\draw [fill] (1,1) circle [radius=0.08];
\draw [fill] (1,0) circle [radius=0.08];
\draw [fill] (0,0) circle [radius=0.08];
\draw [semithick] (1,1) -- (0,1);
\draw [semithick] (1,0) -- (0,1);
\draw [semithick] (1,0) -- (1,1);
\draw [semithick] (0,0) -- (0,1);
\draw [semithick] (0,0) -- (1,0);
\draw [semithick] (0,0) -- (1,1);
\draw (0,1.5);
\end{tikzpicture} & $\left\{ \right\} $ & $\left(0,0,0\right)$ & $\left(0,0,0\right)$ \bigstrut \\
\hline 
\begin{tikzpicture}[scale=0.8]
\draw [fill] (0,1) circle [radius=0.08];
\draw [fill] (1,1) circle [radius=0.08];
\draw [fill] (1,0) circle [radius=0.08];
\draw [fill] (0,0) circle [radius=0.08];
\draw [semithick] (1,0) -- (0,1);
\draw [semithick] (1,0) -- (1,1);
\draw [semithick] (0,0) -- (0,1);
\draw [semithick] (0,0) -- (1,0);
\draw [semithick] (0,0) -- (1,1);
\draw (0,1.5);
\end{tikzpicture} & $\left\{ 1\right\} $ & $\left(1,0,0\right)$ & $\left(1,0,0\right)$ \bigstrut \\
\hline 
\begin{tikzpicture}[scale=0.8]
\draw [fill] (0,1) circle [radius=0.08];
\draw [fill] (1,1) circle [radius=0.08];
\draw [fill] (1,0) circle [radius=0.08];
\draw [fill] (0,0) circle [radius=0.08];
\draw [semithick] (1,1) -- (0,1);
\draw [semithick] (0,0) -- (0,1);
\draw [semithick] (0,0) -- (1,0);
\draw [semithick] (0,0) -- (1,1);
\draw (0,1.5);
\end{tikzpicture} & $\left\{ 2\right\} $ & $\left(2,1,0\right)$ & $\left(1,1,0\right)$ \bigstrut \\
\hline 
\begin{tikzpicture}[scale=0.8]
\draw [fill] (0,1) circle [radius=0.08];
\draw [fill] (1,1) circle [radius=0.08];
\draw [fill] (1,0) circle [radius=0.08];
\draw [fill] (0,0) circle [radius=0.08];
\draw [semithick] (1,1) -- (0,1);
\draw [semithick] (1,0) -- (0,1);
\draw [semithick] (1,0) -- (1,1);
\draw (0,1.5);
\end{tikzpicture} & $\left\{ 3\right\} $ & $\left(3,3,1\right)$ & $\left(1,1,1\right)$ \bigstrut \\
\hline 
\begin{tikzpicture}[scale=0.8]
\draw [fill] (0,1) circle [radius=0.08];
\draw [fill] (1,1) circle [radius=0.08];
\draw [fill] (1,0) circle [radius=0.08];
\draw [fill] (0,0) circle [radius=0.08];
\draw [semithick] (0,0) -- (0,1);
\draw [semithick] (0,0) -- (1,0);
\draw [semithick] (0,0) -- (1,1);
\draw (0,1.5);
\end{tikzpicture} & $\left\{ 1,2\right\} $ & $\left(3,2,0\right)$ & $\left(1,2,0\right)$ \bigstrut \\
\hline 
\begin{tikzpicture}[scale=0.8]
\draw [fill] (0,1) circle [radius=0.08];
\draw [fill] (1,1) circle [radius=0.08];
\draw [fill] (1,0) circle [radius=0.08];
\draw [fill] (0,0) circle [radius=0.08];
\draw [semithick] (1,0) -- (0,1);
\draw [semithick] (1,0) -- (1,1);
\draw (0,1.5);
\end{tikzpicture} & $\text{\ensuremath{\left\{  1,3\right\} } }$ & $\left(4,4,1\right)$ & $\left(1,2,1\right)$ \bigstrut \\
\hline 
\begin{tikzpicture}[scale=0.8]
\draw [fill] (0,1) circle [radius=0.08];
\draw [fill] (1,1) circle [radius=0.08];
\draw [fill] (1,0) circle [radius=0.08];
\draw [fill] (0,0) circle [radius=0.08];
\draw [semithick] (1,1) -- (0,1);
\draw (0,1.5);
\end{tikzpicture} & $\left\{ 2,3\right\} $ & $\left(5,6,2\right)$ & $\left(1,2,2\right)$ \bigstrut \\
\hline 
\begin{tikzpicture}[scale=0.8]
\draw [fill] (0,1) circle [radius=0.08];
\draw [fill] (1,1) circle [radius=0.08];
\draw [fill] (1,0) circle [radius=0.08];
\draw [fill] (0,0) circle [radius=0.08];
\draw (0,1.5);
\end{tikzpicture} & $\left\{ 1,2,3\right\} $ & $\left(6,8,3\right)$ & $\left(1,2,3\right)$ \bigstrut \\
\hline 
\end{tabular}
\medskip
\caption{The correspondence from Theorem~\ref{thm:correspondence} between $\mathcal{T}_3$, $\mathcal{B}_3$, and $\mathcal{A}_3$, where the second column contains the subset of isolated vertices for each graph.}\label{tab:corr}
\end{center}
\end{table}

Throughout the rest of this paper, we will use $\beta(T)$ and $\lambda(T)$ to denote the unique Betti sequence and anti-lecture hall composition associated to each threshold graph $T$ by the correspondence in Theorem~\ref{thm:correspondence}.  Similarly, we will use $T(\beta)$ and $\lambda(\beta)$ to denote the unique threshold graph and anti-lecture hall composition associated to each Betti sequence and $T(\lambda)$ and $\beta(\lambda)$ to denote the unqiue threshold graph and Betti sequence associated to each anti-lecture hall composition.

\section{The Combinatorial Mappings}\label{sec:mappings}

This is the first of two main sections of this paper.  Here, we elaborate on Theorem~\ref{thm:correspondence} by presenting several new explicit combinatorial mappings for the bijections involved.  We begin by reviewing the already established bijection between Betti sequences and anti-lecture hall compositions and establishing some new notation.  We then proceed to introduce the new explicit combinatorial mappings between threshold graphs and Betti sequences and between threshold graphs and anti-lecture hall compositions. 

\subsection{Review and Setup}

The following theorem is proven implicitly in \cite{ES13} where it is described using matrices.  We write it out in explicit combinatorial form below:

\begin{theorem}[Proposition 4.11, \cite{ES13}]\label{thm:transform} If $\beta = (\beta_1,\beta_2,\ldots,\beta_n)$ is the Betti sequence of a quotient by a 2-linear ideal and $\lambda = (\lambda_1,\lambda_2,\ldots,\lambda_n)$ is an anti-lecture hall composition such that $\beta(\lambda) = \beta$ (or equivalently $\lambda(\beta) = \lambda$), then $$ \beta_i = \sum\limits_{k = i}^n \binom{k-1}{i-1} \lambda_k \qquad \text{ and } \qquad \lambda_i = \sum\limits_{k = i}^n (-1)^{i+k} \binom{k-1}{i-1} \beta_k.$$
\end{theorem}

Our aim in Sections~\ref{TG:ALHC} and \ref{TG:BS} is to establish explicit combinatorial mappings between threshold graphs and each of Betti sequences and anti-lecture hall compositions analogous to Theorem~\ref{thm:transform}.  To do this, we show how to calculate the Betti sequence and anti-lecture hall composition associated to a given threshold graph from a prescribed labeling of its non-edges.  For an arbitrary graph $G$, let $\overline{G}$ denote the \emph{complement} of $G$, that is, the graph with $V(\overline{G}) = V(G)$ such that $uv \in E(\overline{G})$ if and only if $uv \notin E(G)$.     

\begin{definition}
Let $T = T(n,\sigma)$ be a threshold graph and, for each $v \in V(T)$, let $$i_T(v) = |\{w \in \sigma \ | \ w > v\}|$$ denote the \emph{number of isolated vertices in $T$ greater than $v$}. The \emph{$\beta$-labeling} of the non-edges of $T$ is the function $\ell_\beta : E(\overline{T}) \to 2^{[n]}$ given by $$\ell_\beta(uv) = \{ u+1, u+2, \ldots, u+1 + i_{T}(v)\}$$ for each $uv \in E(\overline{T})$ with $u < v$ and the \emph{$\lambda$-labeling} of the non-edges of $T$ is the function $\ell_\lambda : E(\overline{T}) \to [n]$ given by $$\ell_{\lambda}(uv) = u + 1 + i_{T}(v)$$ for every $uv \in E(\overline{T})$ with $u < v$.
\end{definition}

\begin{example}\label{ex:graphs}
For the threshold graph, $T = T(5,\{1,2,4\})$, illustrated in Figure~\ref{fig:thresholdgraphexample},  $i_{T}(1) = 2$, $i_{T}(2)  = i_{T}(3)  = 1$, and $i_{T}(4)  = i_{T}(5)  = 0$, so the $\beta$- and $\lambda$-labelings of its non-edges are illustrated in Figure~\ref{fig:ex}. 
\end{example}

\begin{figure}
\begin{center}
\begin{tikzpicture}[scale=1]
\draw [fill] (0,3.6) circle [radius=0.1];
\node [above left] at (0,3.6) {$0$};
\draw [fill] (2.4,3.6) circle [radius=0.1];
\node [above right] at (2.4,3.6) {$1$};
\draw [fill] (3.9,1.8) circle [radius=0.1];
\node [right] at (3.9,1.8) {$ \, 2$};
\draw [fill] (2.4,0) circle [radius=0.1];
\node [below right] at (2.4,0) {$3$};
\draw [fill] (0,0) circle [radius=0.1];
\node [below left] at (0,0) {$4$};
\draw [fill] (-1.5,1.8) circle [radius=0.1];
\node [left] at (-1.5,1.8) {$5 \, $};
\draw [dashed] (0,3.6) -- (2.4,3.6);
\draw [dashed] (0,3.6) -- (3.9,1.8);
\draw [dashed] (0,3.6) -- (0,0);
\draw [dashed] (2.4,3.6) -- (3.9,1.8);
\draw [dashed] (2.4,3.6) -- (0,0);
\draw [dashed] (3.9,1.8) -- (0,0);
\draw [dashed] (2.4,0) -- (0,0);
\node [above] at (1.2,3.6) {\color{gray}$\{1,2,3\}$};
\node [below] at (2.3,2.3) {\color{gray}$\{1,2\}$};
\node [left] at (0,1.8) {\color{gray}$\{1\}$};
\node [above right] at (3.15,2.7) {\color{gray}$\{2,3\}$};
\node [above] at (0.5,1.2) {\color{gray}$\{2\}$};
\node [above] at (1.3,0.7) {\color{gray}$\{3\}$};
\node [above] at (1.9,0) {\color{gray}$\{4\}$};
\end{tikzpicture}
\qquad
\begin{tikzpicture}[scale = 1]
\draw [fill] (0,3.6) circle [radius=0.1];
\node [above left] at (0,3.6) {$0$};
\draw [fill] (2.4,3.6) circle [radius=0.1];
\node [above right] at (2.4,3.6) {$1$};
\draw [fill] (3.9,1.8) circle [radius=0.1];
\node [right] at (3.9,1.8) {$ \, 2$};
\draw [fill] (2.4,0) circle [radius=0.1];
\node [below right] at (2.4,0) {$3$};
\draw [fill] (0,0) circle [radius=0.1];
\node [below left] at (0,0) {$4$};
\draw [fill] (-1.5,1.8) circle [radius=0.1];
\node [left] at (-1.5,1.8) {$5 \, $};
\draw [dashed] (0,3.6) -- (2.4,3.6);
\draw [dashed] (0,3.6) -- (3.9,1.8);
\draw [dashed] (0,3.6) -- (0,0);
\draw [dashed] (2.4,3.6) -- (3.9,1.8);
\draw [dashed] (2.4,3.6) -- (0,0);
\draw [dashed] (3.9,1.8) -- (0,0);
\draw [dashed] (2.4,0) -- (0,0);
\node [above] at (1.2,3.6) {\color{gray}$3$};
\node [below] at (2.6,2.4) {\color{gray}$2$};
\node [left] at (0,1.8) {\color{gray}$1$};
\node [above right] at (3.15,2.7) {\color{gray}$3$};
\node [above left] at (0.8,1.2) {\color{gray}$2$};
\node [above left] at (1.5,0.7) {\color{gray}$3$};
\node [above] at (1.8,0) {\color{gray}$4$};
\end{tikzpicture}
\caption{The $\beta$-labeling (left) and $\lambda$-labeling (right) of the non-edges of the threshold graph $T(5,\{1,2,4\})$ illustrated in Figure~\ref{fig:thresholdgraphexample}.}\label{fig:ex}
\end{center}
 \end{figure}
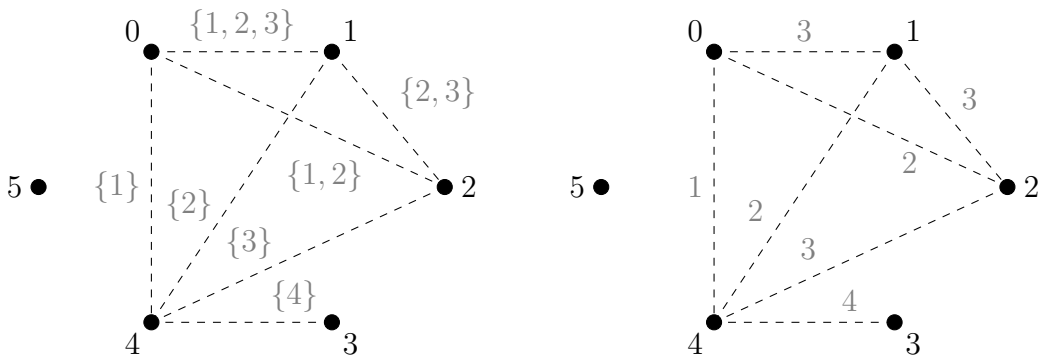 

We conclude this section by extending the recursive characterization of threshold graphs (in terms of dominating and isolating vertices) to Betti sequences and anti-lecture hall compositions.  

\begin{definition}
For a threshold graph $T = T(n,\sigma)$, let $T_\delta = T(n+1,\sigma)$ be the threshold graph obtained by appending a dominating vertex to $T$ and $T_\iota = T(n+1, \sigma \cup \{n+1\})$ be the threshold graph obtained by appending an isolated vertex to $T$.  
\end{definition}

Engstr\"om and Stamps showed how to express the Betti sequences of $T_\delta$ and $T_\iota$ in terms of the Betti sequence of $T$ in \cite{ES13}.

\begin{lemma}[Proposition 4.2, \cite{ES13}]\label{lem:betti-shift} If $T$ is a threshold graph on $n+1$ vertices with Betti sequence $\beta(T) = \left(\beta_{1},\dots,\beta_{n}\right)$, then
\begin{enumerate} \itemsep 10pt
\item $\beta(T_\delta) = \left(\beta_{1},\dots,\beta_{n},0\right)$ and
\item $\beta(T_\iota) = \left(\beta_{1},\dots,\beta_{n},0\right) + \left(0,\beta_{1},\dots,\beta_{n}\right) + \left(\binom{n+1}{1},\dots,\binom{n+1}{n+1} \right).$
\end{enumerate}
\end{lemma}

There is an analogous result to Lemma~\ref{lem:betti-shift} for anti-lecture hall compositions.  

\begin{lemma}
\label{lem:alhc-shift} If $T$ is a threshold graph on $n+1$ vertices with anti-lecture hall composition $\lambda(T) =\left(\lambda_{1},\dots,\lambda_{n}\right)$, then
\begin{enumerate} \itemsep 10pt
\item $\lambda(T_\delta) = \left(\lambda_{1},\dots,\lambda_{n},0\right)$ and  
\item $\lambda(T_\iota) = \left(1,\lambda_{1}+1,\dots,\lambda_{n}+1\right)$.
\end{enumerate}
\end{lemma}

We will use the following straightforward binomial identity in our proof of Lemma~\ref{lem:alhc-shift}. 

\begin{identity}\label{id:binom}
For every $1 \leq k \leq n$, $$\sum\limits_{i = k}^n (-1)^{i+k} \binom{n}{i} \binom{i-1}{k-1} = 1.$$
\end{identity}



\begin{proof}[Proof of Lemma~\ref{lem:alhc-shift}.]
Let $T$ be a threshold graph on $n+1$ vertices with $\lambda(T) = (\lambda_1,\ldots,\lambda_n)$ and $\beta(T) = ( \beta_1,\ldots,\beta_n)$.  By Theorem~\ref{thm:transform}, $$\lambda_{i} = \sum_{k=i}^{n}\left(-1\right)^{i+k}\binom{k-1}{i-1}\beta_{k}.$$

Let us first consider the threshold graph $T_\delta$.  By Lemma \ref{lem:betti-shift}, $\beta_k(T_\delta) = \beta_k$ for $1 \leq k \leq n$ and $\beta_{n+1}(T_\delta) = 0$.  So, by Theorem~\ref{thm:transform}, $$\lambda_{i}(T_\delta) = \sum_{k=i}^{n+1}\left(-1\right)^{i+k}\binom{k-1}{i-1}\beta_{k}(T_\delta) = \sum_{k=i}^{n}\left(-1\right)^{i+k}\binom{k-1}{i-1}\beta_{k} = \lambda_i$$ for $1 \leq i \leq n$ and $$\lambda_{n+1}(T_\delta) = \beta_{n+1}(T_\delta) = 0.$$  Thus, $\lambda(T_\delta) = (\lambda_1,\ldots,\lambda_n,0)$.  

Next, we consider the threshold graph $T_\iota$.  By Lemma \ref{lem:betti-shift}, $$\beta_{k}\left(T_{\iota}\right)=\begin{cases} \beta_{1}+\binom{n+1}{1} & \text{for }k=1,\\[0.5em] \beta_{k}+\beta_{k-1}+\binom{n+1}{k} & \text{for } 2 \leq k \leq n,\\[0.5em] \beta_{n}+\binom{n+1}{n+1} & \text{for } k = n+1\end{cases}$$  so, by Theorem~\ref{thm:transform}, \begin{align*} \lambda_1(T_\iota) &= \sum_{k=1}^{n+1}\left(-1\right)^{k+1} \beta_{k}(T_\iota) \\[0.5em] &= \sum_{k=1}^{n}\left(-1\right)^{k+1} \beta_{k}+ \sum_{k=2}^{n+1}\left(-1\right)^{k+1} \beta_{k-1} + \sum_{k=1}^{n+1}\left(-1\right)^{k+1} \binom{n+1}{k} \\[0.5em] &= \sum_{k=1}^{n}\left(-1\right)^{k+1} \beta_{k}+ \sum_{k=1}^{n}\left(-1\right)^{k} \beta_{k} + \sum_{k=1}^{n+1}\left(-1\right)^{k+1} \binom{n+1}{k} \\[0.5em] &= 0 + \sum_{k=1}^{n+1}\left(-1\right)^{k+1} \binom{n+1}{k} \\[0.5em] &= 1 \end{align*}
and, for $i \geq 2$, \begin{align*} \lambda_{i}(T_\iota) &= \sum_{k=i}^{n+1}\left(-1\right)^{i+k}\binom{k-1}{i-1}\beta_{k}(T_\iota) \\[0.5em] &= \sum_{k=i}^{n}\left(-1\right)^{i+k}\binom{k-1}{i-1}\beta_{k} + \sum_{k=i}^{n+1}\left(-1\right)^{i+k}\binom{k-1}{i-1}\beta_{k-1} + \sum\limits_{k = i}^{n+1} (-1)^{i+k} \binom{k-1}{i-1}\binom{n+1}{i} \\[0.5em] &= \sum_{k=i}^{n}\left(-1\right)^{i+k}\binom{k-1}{i-1}\beta_{k} + \sum_{k=i-1}^{n}\left(-1\right)^{i+k-1}\binom{k}{i-1}\beta_{k} + 1 \\[0.5em] &= \sum_{k=i-1}^{n}\left(-1\right)^{i+k-1}\binom{k-1}{i-2}\beta_{k} + 1 \\[0.5em] &= \lambda_{i-1} + 1 \end{align*} where the third equality follows from reindexing the middle summand and applying Identity~\ref{id:binom} to the third summand.  Thus, $\lambda(T_\iota) = (1,1+\lambda_1,\ldots,1+\lambda_n)$. 
\end{proof}

Table \ref{tab:betti-alhc} lists the Betti sequences and anti-lecture hall compositions associated to the threshold graphs for $1 \leq n \leq 4$ according to the steps defined in Lemmas \ref{lem:betti-shift} and \ref{lem:alhc-shift} where upper cells correspond to appending dominating vertices and lower cells correspond to appending isolated vertices.  

\begin{table}
\renewcommand{\arraystretch}{1.25}
\setlength{\tabcolsep}{5pt}
\begin{multicols}{2}
\begin{center}
\begin{tabular}{|c|c|c|c|}
\hline 
  $\mathcal{B}_1$ & $\mathcal{B}_2$ & $\mathcal{B}_3$ & $\mathcal{B}_4$  \tabularnewline
\hline 
\hline 
 \multirow{8}{*}{$\left(0\right)$} & \multirow{4}{*}{$\left(0,0\right)$} & \multirow{2}{*}{$\left(0,0,0\right)$} & $\left(0,0,0,0\right)$ \tabularnewline
\cline{4-4} 
  &  &  & $\left(4,6,4,1\right)$ \tabularnewline
\cline{3-4} 
  &  & \multirow{2}{*}{$\left(3,3,1\right)$} & $\left(3,3,1,0\right)$ \tabularnewline
\cline{4-4} 
  &  &  & $\left(7,12,8,2\right)$ \tabularnewline
\cline{2-4} 
  & \multirow{4}{*}{$\left(2,1\right)$} & \multirow{2}{*}{$\left(2,1,0\right)$} & $\left(2,1,0,0\right)$ \tabularnewline
\cline{4-4} 
  &  &  & $\left(6,9,5,1\right)$ \tabularnewline
\cline{3-4} 
  &  & \multirow{2}{*}{$\left(5,6,2\right)$} & $\left(5,6,2,0\right)$ \tabularnewline
\cline{4-4} 
  &  &  & $\left(9,17,12,3\right)$ \tabularnewline
\cline{1-4} 
 \multirow{8}{*}{$\left(1\right)$} & \multirow{4}{*}{$\left(1,0\right)$} & \multirow{2}{*}{$\left(1,0,0\right)$} & $\left(1,0,0,0\right)$ \tabularnewline
\cline{4-4} 
  &  &  & $\left(5,7,4,1\right)$ \tabularnewline
\cline{3-4} 
  &  & \multirow{2}{*}{$\left(4,4,1\right)$} & $\left(4,4,1,0\right)$ \tabularnewline
\cline{4-4} 
  &  &  & $\left(8,14,9,2\right)$ \tabularnewline
\cline{2-4} 
  & \multirow{4}{*}{$\left(3,2\right)$} & \multirow{2}{*}{$\left(3,2,0\right)$} & $\left(3,2,0,0\right)$ \tabularnewline
\cline{4-4} 
  &  &  & $\left(7,11,6,1\right)$ \tabularnewline
\cline{3-4} 
  &  & \multirow{2}{*}{$\left(6,8,3\right)$} & $\left(6,8,3,0\right)$ \tabularnewline
\cline{4-4} 
   &  &  & $\left(10,20,15,4\right)$ \tabularnewline
\hline 
\end{tabular}

\begin{tabular}{|c|c|c|c|}
\hline 
$\mathcal{A}_1$ & $\mathcal{A}_2$ & $\mathcal{A}_3$ & $\mathcal{A}_4$  \tabularnewline
\hline 
\hline 
 \multirow{8}{*}{$\left(0\right)$} & \multirow{4}{*}{$\left(0,0\right)$} & \multirow{2}{*}{$\left(0,0,0\right)$} & $\left(0,0,0,0\right)$ \tabularnewline
\cline{4-4} 
  &  &  & $\left(1,1,1,1\right)$ \tabularnewline
\cline{3-4} 
 &  & \multirow{2}{*}{$\left(1,1,1\right)$} & $\left(1,1,1,0\right)$ \tabularnewline
\cline{4-4} 
  &  &  & $\left(1,2,2,2\right)$ \tabularnewline
\cline{2-4} 
  & \multirow{4}{*}{$\left(1,1\right)$} & \multirow{2}{*}{$\left(1,1,0\right)$} & $\left(1,1,0,0\right)$ \tabularnewline
\cline{4-4} 
  &  &  & $\left(1,2,2,1\right)$ \tabularnewline
\cline{3-4} 
  &  & \multirow{2}{*}{$\left(1,2,2\right)$} & $\left(1,2,2,0\right)$ \tabularnewline
\cline{4-4} 
  &  &  & $\left(1,2,3,3\right)$ \tabularnewline
\cline{1-4} 
  \multirow{8}{*}{$\left(1\right)$} & \multirow{4}{*}{$\left(1,0\right)$} & \multirow{2}{*}{$\left(1,0,0\right)$} & $\left(1,0,0,0\right)$ \tabularnewline
\cline{4-4} 
  &  &  & $\left(1,2,1,1\right)$ \tabularnewline
\cline{3-4} 
  &  & \multirow{2}{*}{$\left(1,2,1\right)$} & $\left(1,2,1,0\right)$ \tabularnewline
\cline{4-4} 
   &  &  & $\left(1,2,3,2\right)$ \tabularnewline
\cline{2-4} 
  & \multirow{4}{*}{$\left(1,2\right)$} & \multirow{2}{*}{$\left(1,2,0\right)$} & $\left(1,2,0,0\right)$ \tabularnewline
\cline{4-4} 
   &  &  & $\left(1,2,3,1\right)$ \tabularnewline
\cline{3-4} 
   &  & \multirow{2}{*}{$\left(1,2,3\right)$} & $\left(1,2,3,0\right)$ \tabularnewline
\cline{4-4} 
  &  &  & $\left(1,2,3,4\right)$ \tabularnewline
\hline 
\end{tabular}
\end{center}
\end{multicols}
\caption{Betti sequences and anti-lecture hall compositions of threshold graphs.}\label{tab:betti-alhc}
\end{table}


We are now ready to state and prove the explicit bijections between threshold graphs and Betti sequences and between threshold graphs and anti-lecture hall compositions.  

\subsection{Threshold Graphs and Anti-Lecture Hall Compositions}\label{TG:ALHC}

We begin by showing how to move directly between threshold graphs and anti-lecture hall compositions without having to pass through the Betti sequences.    

\begin{theorem}\label{thm:threshold-to-alhc}
If $T$ is a threshold graph on $n+1$ vertices, then $$\lambda_{k}(T) = \left|E^\lambda_k(\overline{T})\right|$$ where $E^\lambda_k(\overline{T}) = \{uv \notin E(T) \ | \ \ell_\lambda(uv) = k\}$.  
\end{theorem}

\begin{proof}[Proof of Theorem~\ref{thm:threshold-to-alhc}.]
We prove that the formula satisfies the recurrences in Lemma~\ref{lem:alhc-shift}.  There are two initial cases for $n = 1$. For the graph $T = T(1,\emptyset)$, $E^\lambda_1(\overline{T}) = \emptyset$; so the formula gives $\lambda_1(T) = 0$.  For the graph $T = T(1,\{1\})$, $i_T(0) = 1$, $i_T(1) = 0$, and $E^\lambda_1(\overline{T}) = \{\{0,1\}\}$; so the formula gives $\lambda_1(T) = 1$. Now suppose that the formula is true for threshold graph $T = T(n,\sigma)$.  We will show that the formula holds for $T_\delta$ and $T_\iota$.    

For $T_\delta$, observe that $i_{T_\delta}(v) = i_{T}(v)$ for $1 \leq v \leq n$, $i_{T_\delta}(n+1) = 0$, and $E(\overline{T_\delta}) = E(\overline{T})$.  It follows that $E^\lambda_k(\overline{T_\delta}) = E^\lambda_k(\overline{T})$ for $1 \leq k \leq n$ and $E^\lambda_{n+1}(\overline{T_{\delta}}) = \emptyset$. Hence the formula gives $\lambda_k(T_\delta) = \lambda_k(T)$ for $1 \leq k \leq n$ and $\lambda_{n+1}(T_\delta) = 0$.  

For $T_\iota$, observe that $i_{T_\iota}(v) = i_{T}(v)+1$ for $1 \leq v \leq n$, $i_{T_\iota}(n+1) = 0$, and $$E(\overline{T_\iota}) = E(\overline{T}) \cup \{\{i,n+1\} \ | \ 0 \leq i \leq n\}.$$  Since $\ell_\lambda(\{i,n+1\}) = i+1$, it follows that $E^\lambda_1(\overline{T_\iota}) = \{0,n+1\}$ and $$E^\lambda_k(\overline{T_\iota}) = E^\lambda_{k-1}(\overline{T}) \cup \{\{k-1,n+1\}\}$$ for $2 \leq k \leq n+1$. Hence $\lambda_1(T_\iota) = 1$ and $\lambda_k(T_\iota) = \lambda_{k-1}(T) + 1$ for $2 \leq k \leq n+1$.  
\end{proof}

\begin{example}
For the threshold graph $T = T(5,\{1,2,4\})$ in Figure~\ref{fig:thresholdgraphexample}, recall that $i_T(1) = 2$, $i_T(2) = i_T(3) = 1$, and $i_T(4) = i_T(5) = 0$ and note that \begin{align*} E^\lambda_1(\overline{T}) &= \{04\}, \\[0.5em] E^\lambda_2(\overline{T}) &= \{12,14\}, \\[0.5em] E^\lambda_3(\overline{T}) &= \{01,12, 24\}, \\[0.5em] E^\lambda_4(\overline{T}) &= \{34\}, \\[0.5em] E^\lambda_5(\overline{T}) &= \emptyset\end{align*} from the $\lambda$-labeling of $\overline{T}$ in Figure~\ref{fig:ex}.  Theorem~\ref{thm:threshold-to-alhc} asserts that $\lambda(T) = (1,2,3,1,0)$, which is indeed the case. 
\end{example}

For the reverse direction, we can quickly calculate the dominating (or isolated) vertices of the threshold graph associated to a given anti-lecture hall composition.  

\begin{theorem}\label{thm:differences}
If $\lambda = (\lambda_1, \ldots, \lambda_n)$ is an anti-lecture hall composition of length $n$ bounded above by $1$, then $T(\lambda) = T(n,\sigma)$ where $\sigma$ is the complement in $\{1,\ldots,n\}$ of the nonzero elements in $\{ k - \lambda_k \ | \ 1 \leq k \leq n\}.$ 
\end{theorem}

\begin{proof}
Let $\tau(T) = \{k - \lambda_k(T) \ | \ 1 \leq k \leq n\}$ for an arbitrary threshold graph $T$.  We want to show that $\tau(T(n,\sigma)) = [n] \setminus \sigma$ for every $n \in \mathbb{N}$ and $\sigma \subseteq [n]$ and proceed by induction on $n$. For the base cases ($n=1$), observe that $T(1,\emptyset)$ is the threshold graph associated to the anti-lecture hall composition $(0)$ and $T(1,\{1\})$ is the threshold graph associated to $(1)$. 

For the induction step, let $n \geq 2$, suppose that $\tau(T(m,\sigma)) = [m] \setminus \sigma$ for every $m < n$ and $\sigma \subseteq [m]$, let $\lambda = (\lambda_1,\ldots,\lambda_n)$ be an anti-lecture hall composition, let $T = T(n,\sigma)$ be its associated threshold graph.  If $\lambda_n = 0$, then $n = n - \lambda_n \in \tau(T)$ and $T = T_\delta^\prime$ for $T^\prime = T(n-1,\sigma)$.  By the induction hypothesis, $\tau(T^\prime) = [n-1] \setminus \sigma$, so $$\tau(T) = \{n\} \cup [n-1] \setminus \sigma = [n] \setminus \sigma.$$  If $\lambda_n \neq 0$, then $n \in \sigma$ and $T = T_{\iota}^\prime$ for $T^\prime = T(n-1,\sigma \setminus \{n\})$, which means $\lambda_1 = 1$ and $\lambda_k = 1 + \lambda_{k-1}(T^\prime)$ for $2 \leq k \leq n$.  By the induction hypothesis, $\tau(T^\prime) = [n-1] \setminus (\sigma \setminus \{n\}) = [n] \setminus \sigma$.  Since $1 - \lambda_1 = 0$ and $k - \lambda_k = k - (1+\lambda_{k-1}(T^\prime)) = (k-1) - \lambda_{k-1}(T^\prime)$ for $2 \leq k \leq n$, it follows that $\tau(T) = \tau(T^\prime) = [n] \setminus \sigma$.
\end{proof}

\begin{example}
For the anti-lecture hall composition $\lambda = (1,2,3,1,0)$, note that $1-1 = 0$, $2-2 = 0$, $3-3 = 0$, $4 - 1 = 3$, and $5-0 = 5$.  Theorem~\ref{thm:differences} asserts that the dominating vertices of the threshold graph $T$ associated to $\lambda$ are $3$ and $5$. Hence $T = T(5,\{1,2,4\})$, which is indeed the case.
\end{example}

\subsection{Threshold Graphs and Betti Sequences}\label{TG:BS}

We now consider an explicit relationship between threshold graphs and Betti sequences. Dochtermann and Engstr\"om \cite{DE09} proved the following formula for the Betti sequence of a chordal graph.

\begin{theorem}[Theorem 3.2, \cite{DE09}]\label{thm:DE}
If $G$ is a chordal graph, then
$$\beta_{k}\left(\Bbbk\left[G\right]\right)=\sum_{W\in\binom{V\left(G\right)}{k+1}}\left(-1+\#\text{ components of }G\left[W\right]\right)$$ where the sum runs over every $(k+1)$-element subset $W$ of $V(G)$.
\end{theorem}

The formula in Theorem~\ref{thm:DE}, which applies to threshold graphs since threshold graphs are chordal, is combinatorial in nature and bijective when restricted to threshold graphs \cite{ES13}. However, it can be rather cumbersome to work with since calculating the full Betti sequence involves summing over all the subsets of $V(G)$, and counting connected components of induced subgraphs is not always straightforward.  Below, we introduce a new purely combinatorial formula for the Betti sequence of a threshold graph $T$ based on the $\beta$-labeling of $\overline{T}$, which also has the advantage that the summation is over the non-edges of $T$.  

\begin{theorem}\label{thm:Betti-direct}
If $T$ is a threshold graph on $n+1$ vertices, then $$\beta_k(T) = \sum\limits_{uv \in E^\beta_k(\overline{T})} \binom{v}{u+1} \binom{i_{T}(v)}{k - u - 1}$$ where $E^\beta_k(\overline{T}) = \{ uv \notin E(T) \ | \ k \in \ell_\beta(uv) \}$. 
\end{theorem}

\begin{proof}[Proof of Theorem~\ref{thm:Betti-direct}]
We prove that the formula satisfies the recursive formulae in Lemma~\ref{lem:betti-shift}.  There are two initial cases for $n =1$.  For the graph $T = T(1,\emptyset)$, $E^\beta_1(\overline{T}) = \emptyset$; so the proposed formula gives $\beta_1(T) = 0$.  For the graph $T = T(1,\{1\})$, $i_T(0) = 1$, $i_T(1) = 0$, and $E^\beta_1(\overline{T}) = \{\{0,1\}\}$; so the proposed formula gives $\beta_1(T) = \binom{1}{1}\binom{0}{0} = 1$.  Now suppose the proposed formula is true for threshold graph $T = T(n,\sigma)$; we will show that it holds for $T_\delta$ and $T_\iota$ as well.  

For $T_\delta$, observe that $i_{T_\delta}(v) = i_{T}(v)$ for $1 \leq v \leq n$, $i_{T_\delta}(n+1) = 0$, and $E(\overline{T_\delta}) = E(\overline{T})$.  It follows that $E^\beta_k(\overline{T_\delta}) = E^\beta_k(\overline{T})$ for $1 \leq k \leq n$ and $E^\beta_{n+1}(\overline{T_{\delta}}) = \emptyset$.  Thus, the formula gives $\beta_k(T_\delta) = \beta(T)$ for $1 \leq k \leq n$ and $\beta_{n+1}(T_\delta) = 0$.  

For $T_\iota$, observe that $i_{T_\iota}(v) = i_{T}(v)+1$ for $1 \leq v \leq n$ and $i_{T_\iota}(n+1) = 0$.
For $k = 1$, $E^\beta_1(\overline{T_\iota}) = E^\beta_1(\overline{T}) \cup \big\{\{0,n+1\}\big\}$, which means $$\beta_1(T_\iota) = \beta_1(T)+\binom{n+1}{1} \binom{0}{0} = \beta_1(T) + \binom{n+1}{1}.$$  For $2 \leq k \leq n$, $$E^\beta_k(\overline{T_\iota}) = E^\beta_k(\overline{T}) \cup E^\beta_{k-1}(\overline{T}) \cup \big\{ \{k-1,n+1\} \big\}.$$    If $\{u,v\} \in E^\beta_k(\overline{T}) \setminus E^\beta_{k-1}(\overline{T})$, then $u = k-1$ and $$\displaystyle \binom{v}{u+1} \binom{i_{T_\iota}(v)}{0} = \binom{v}{u+1} \binom{i_{T}(v)}{0}.$$  If $\{u,v\} \in E^\beta_{k-1}(\overline{T}) \setminus E^\beta_{k}(\overline{T})$, then $i_T(v) = k - u - 2$ and $$\displaystyle \binom{v}{u+1} \binom{k - u - 1}{k - u - 1} = \binom{v}{u+1} \binom{k-u-2}{k-u-2}.$$ If $\{u,v\} \in E^\beta_k(\overline{T}) \cap E^\beta_{k-1}(\overline{T})$, then $$\binom{v}{u+1} \binom{i_{T_\iota}(v)}{k-u-1} = \binom{v}{u+1} \left( \binom{i_T(v)}{k-u-1} + \binom{i_T(v)}{k-u-2}\right).$$ Bringing this all together we see that the formula satisfies $$\beta_k(T_{\iota}) = \beta_k(T) + \beta_{k-1}(T) + \binom{n+1}{k}$$ for $2 \leq k \leq n$.  Finally, for $k = n+1$, $E^\beta_{n+1}(\overline{T_{\iota}}) = E^\beta_n(\overline{T}) \cup \big\{\{n,n+1\}\big\}$, so $$\beta_{n+1}(T_\iota) = \beta_n(T) + \binom{n+1}{n+1} \binom{0}{0} = \beta_n(T) + \binom{n+1}{n+1}.$$  This completes the proof.
\end{proof}

\begin{example}
For the threshold graph $T = T(5,\{1,2,4\})$ in Figure~\ref{fig:thresholdgraphexample}, recall that $i_T(1) = 2$, $i_T(2) = i_T(3) = 1$, and $i_T(4) = i_T(5) = 0$ and note that 
\begin{align*} E^\beta_1(\overline{T}) &= \{01,02,04\} \\[0.5em] E^\beta_2(\overline{T}) &= \{01,02,12,14\} \\[0.5em] E^\beta_3(\overline{T}) &= \{01,12,24\} \\[0.5em] E^\beta_4(\overline{T}) &= \{34\} \\[0.5em] E^\beta_5(\overline{T}) &= \emptyset \end{align*} from the $\beta$-labeling of $\overline{T}$ in Figure~\ref{fig:ex}.  Theorem~\ref{thm:Betti-direct} asserts that \begin{align*} \beta_1(T) &= \binom{1}{1} \binom{2}{0} + \binom{2}{1} \binom{1}{0} + \binom{4}{1} \binom{0}{0} = 7 \\[1em] \beta_2(T) &= \binom{1}{1} \binom{2}{1} + \binom{2}{1} \binom{1}{1} + \binom{2}{2} \binom{1}{0} + \binom{4}{2} \binom{0}{0} = 11 \\[1em] \beta_3(T) &= \binom{1}{1} \binom{2}{2} + \binom{2}{2} \binom{1}{1} + \binom{4}{3} \binom{0}{0} = 6 \\[1em] \beta_4(T) &= \binom{4}{4} \binom{0}{0} = 1 \\[1em] \beta_5(T) &= 0, \end{align*} which is indeed the Betti sequence associated to $T$.  
\end{example}

Finally, we show how one can determine the threshold graph associated to a given Betti sequence.  Since appending a dominating vertex to a graph corresponds to appending a $0$ to its corresponding Betti sequence, and appending an isolated vertex results in a Betti sequence whose last entry is nonzero, one can read off the generating sequence of the threshold graph corresponding to a given by sequence by sequentially inverting the steps in Lemma~\ref{lem:betti-shift}, as demonstrated in Example~4.14 of \cite{ES13}. Here, we give a more direct way to determine the threshold graph associated to a given anti-lecture hall composition, analogous to Theorem~\ref{thm:differences}.  
 
\begin{theorem}\label{thm:alhc-differences}
If $\beta = (\beta_1, \ldots, \beta_n)$ is the Betti vector of a $2$-linear ideal, then $T(\beta) = T(n,\sigma)$ where $\sigma$ is the complement in $\{1,\ldots,n\}$ of the set of nonzero elements of $$\left\{i - \sum\limits_{k = i}^n (-1)^{i+k} \binom{k-1}{i-1} \beta_k \ \Bigg| \ 1 \leq i \leq n \right\}.$$ 
\end{theorem}

\begin{proof}
This is a straightforward consequence of Theorems \ref{thm:transform} and \ref{thm:differences}.
\end{proof}

\begin{example}
For the Betti sequence $\beta = (7,11,6,1,0)$, note that \begin{align*} 1 - 1 \cdot 7 + 1 \cdot 11 - 1 \cdot 6 + 1 \cdot 1 - 1 \cdot 0 &= 0 \\[0.5em] 2 - 1 \cdot 11 + 2 \cdot 6 - 3 \cdot 1 + 4 \cdot 0 &= 0 \\[0.5em] 3 - 1 \cdot 6 + 3 \cdot 1 - 3 \cdot 0 &= 0 \\[0.5em] 4 - 1 \cdot 1 + 4 \cdot 0 &= 3 \\[0.5em] 5 - 1 \cdot 0 & = 5.\end{align*}  Theorem~\ref{thm:alhc-differences} asserts that the dominating vertices of the threshold graph $T$ associated to $\beta$ are $3$ and $5$. Hence $T = T(5,\{1,2,4\})$, which is indeed the case.
\end{example}

\section{Probabilistic Results}

Having established explicit relationships between threshold graphs, Betti sequences, and anti-lecture hall compositions, it is natural to consider what these various objects might look like when a random model is imposed on them. In this second main section of this paper, we apply the bijections in the previous section to a random model for threshold graphs, thereby constructing associated random models for Betti sequences and anti-lecture hall compositions.  With these, we are able to calculate the expected Betti numbers and anti-lecture hall compositions for a given $n \in \mathbb{N}$ with respect to a $1$-parameter family of distributions that includes the uniform distribution.  

\subsection{Random Threshold Graphs}

Random threshold graphs have been studied extensively, for instance by Reilly and Scheinermann \cite{RS09}, where they establish a wide array of expected values for properties and parameters including connectivity, maximum and minimum degrees, degree sequences, Laplacian eigenvalues, chromatic and clique numbers, cyclicity, and Hamiltonicity, to name a few.  Even more impressive than the breadth of properties considered is that the above results are all exact (as opposed to asymptotic, which is often the case for analogous results, say on Erd\H{o}s-R\'enyi random graphs).  Random threshold graphs also played a central role in some early works on graph limits, see the paper of Diaconis, Holmes, and Janson \cite{DHS08}.  

There are several natural and equivalent random models for threshold graphs.  We define them according to the model introduced by Hagberg, Swart, and Schult \cite{HSS06}.  

\begin{definition}
Given $n \in \mathbb{N}$ and $p \in [0,1]$, let $T(n,p)$ denote the threshold graph $T(n,\sigma)$ where the probability that $i \in \sigma$ is equal to $p$ for each $i \in \{1,\ldots,n\}$.
\end{definition}

We note that $T(n,1/2)$ corresponds to the uniform distribution on $\mathcal{T}_n$ since the probability of observing $T(n,\sigma)$ in $T(n,p)$ is $p^{|\sigma|}(1-p)^{n-|\sigma|}$, which is $\frac{1}{2^n} = \frac{1}{|\mathcal{T}_n|}$ when $p = 1/2$. This result has appeared several places in the literature, for instance see Section~2 and Corollary 6.6 of \cite{DHS08}.  

\begin{lemma}\label{lem:uniform}
For every $n \in \mathbb{N}$, $T(n,1/2)$ yields the uniform distribution on $\mathcal{T}_n$. \qed
\end{lemma}

For $p = 0$, $T(n,0)$ is necessarily the complete graph on $n+1$ vertices and, for $p = 1$, $T(n,1)$ is necessarily the graph on $n+1$ vertices with no edges.  Since we know the Betti sequence and anti-lecture hall composition associated to each of these graphs, we will restrict our focus to $0 < p < 1$ for most of the results in this section. 

\subsection{Random Betti Sequences}

We first consider the Betti numbers, calculating the expected projective dimension and Betti sequence of $\Bbbk[T(n,p)]$.  

\subsubsection{Expected Projective Dimension}

Recall that the projective dimension of $\Bbbk\left[T\right]$ is the largest index $k$ for which $\beta_k(\Bbbk[T])$ is nonzero.  

\begin{proposition}\label{prop:probdim}
If $1 \leq m \leq n$ and $0 < p < 1$, then $$\mathbb{P}(\dim_{\mathrm{proj}}\Bbbk\left[T(n,p)\right] = m) = p (1-p)^{n-m}.$$
\end{proposition}

\begin{proof}
By Lemma~\ref{lem:betti-shift}, the projective dimension of $\Bbbk[T(n,\sigma)]$ is equal to the largest value $m$ in $\sigma$. The probability that $m \in \sigma$ is $p$, and the probability that each value $i\in\left\{m+1,\dots,n\right\}$ is not in $\sigma$ is $1-p$. Thus, the probability that $m$ is the largest value in $\sigma$ is $p(1-p)^{n-m}$.
\end{proof}

With this, we can give a formula for the expected projective dimension of $\Bbbk[T(n,p)]$.

\begin{proposition}\label{prop:expecteddim}
\label{prop:e(dimproj)}  The expected projected dimension of a random threshold graph, $T(n,p)$, with $0 < p < 1$ is $$\mathbb{E}\left[\dim_{\mathrm{proj}}\left(\Bbbk\left[T\right]\right)\right] = n + \frac{(1-p)^{n+1} - (1-p)}{p}.$$
\end{proposition}

\begin{proof}
By Proposition~\ref{prop:probdim}, the probability that $\dim_{\mathrm{proj}}\Bbbk\left[T(n,p)\right] = m$ is $p (1-p)^{n-m}.$ Thus, the expected value for the projective dimension of $\Bbbk\left[T\right]$ is given by
\begin{align*}
\mathbb{E}\left[\dim_{\mathrm{proj}}\left(\Bbbk\left[T\right]\right)\right] &= \sum_{m=0}^n m \cdot \mathbb{P}(\dim_{\mathrm{proj}}\Bbbk\left[T\right] = m) \\[1em]
 &= \sum_{m=0}^{n} m \cdot p (1-p)^{n-m} \\[0.5em]
 &= p(1-p)^{n-1} \sum\limits_{m = 0}^n m \left( \frac{1}{1-p} \right)^{m-1} \\[1em]
 &= p(1-p)^{n-1} \frac{1 - (n+1)\left(\frac{1}{1-p}\right)^n + n\left(\frac{1}{1-p}\right)^{n+1}}{\left(1-\frac{1}{1-p}\right)^2} \\[1em]
 &= \frac{ (1-p)^{n+1} - (n+1)(1-p) + n}{p}  \\[1em]
 &= \frac{(1-p)^{n+1} - (1 - p)}{p} + n
 \end{align*}
 where the fourth equality follows from Lemma~\ref{lem:comb} below by setting $\displaystyle q = \frac{1}{1-p}$.
\end{proof}

\begin{lemma}\label{lem:comb}
For every $n \in \mathbb{N}$ and indeterminate $q$, $$\sum\limits_{m = 0}^n mq^{m-1} = \frac{1 - (n+1)q^n + nq^{n+1}}{(1-q)^2}.$$  
\end{lemma}

\begin{proof}
Let $n \in \mathbb{N}$, then $$ \sum\limits_{m = 0}^n mq^{m-1} = \frac{d}{dq} \left ( \sum\limits_{m = 0}^n q^m \right) = \frac{d}{dq} \left (\frac{1 - q^{n+1}}{1-q} \right) = \frac{1 - (n+1)q^n + nq^{n+1}}{(1-q)^2}.$$
\end{proof}

An immediate consequence of Lemma~\ref{lem:uniform} and Proposition~\ref{prop:expecteddim} is a simple formula for the expected projective dimension of the quotient ring of a threshold graph sampled uniformly at random.

\begin{corollary}
The expected projective dimension of $\Bbbk\left[T(n,1/2)\right]$
is $$\mathbb{E}\left[\dim_{\mathrm{proj}}\left(\Bbbk\left[T(n,1/2)\right]\right)\right] = n - 1 + \left(1/2\right)^{n}.$$
\end{corollary}

\subsubsection{Expected Betti Numbers}

Next, we compute the expected Betti sequence associated to $T(n,p)$.  To simplify notation, let $B_k(n,p) = \mathbb{E}(\beta_k(T(n,p))$.  We begin by establishing a recurrence on $B_k(n,p)$.  

\begin{proposition}\label{prop:newbetti}
If we set $B_k(n,p) = 0$ for $k = 0$ and $k > n$, then the triangle of numbers $B_k(n,p)$ satisfies the recurrence  $$B_k(n,p) = B_k(n-1,p) + p \cdot B_{k-1}(n-1,p) + p \cdot \binom{n}{k}$$ for $1 \leq k \leq n$.
\end{proposition}

\begin{proof}
This follows from Lemma~\ref{lem:betti-shift} since every element of $\mathcal{B}_n$ arises from an element of $\mathcal{B}_{n-1}$ via one of the $\delta$ or $\iota$ operations.  We know that elements of $\mathcal{B}_n$ arising from $\delta$ correspond to the same compositions with $0$ appended to their ends and that each will occur with probability $1-p$.  Thus, these contribute $(1-p) \cdot B_k(n-1,p)$ to $B_k(n,p)$.  The remaining elements of $\mathcal{B}_n$ arise from applying $\iota$ to each element of $\mathcal{B}_{n-1}$.  This corresponds to adding a shifted copy of the sequence to itself along with a sequence of binomial coefficients, and occurs with probability $p$ in each case.  Together they contribute $$p \cdot \left(B_k(n-1,p) + B_{k-1}(n-1,p) + \binom{n}{k}\right)$$ to $B_k(n,p)$.  The result follows from combining like terms.
\end{proof}

With this, we can give a formula for the expected Betti numbers of $\Bbbk[T(n,p)]$.  

\begin{theorem}\label{thm:newbetti}
For every $1 \leq k \leq n$ and $0 < p < 1$, $$\mathbb{E}(\beta_k(\Bbbk[T(n,p)])) = \binom{n+1}{k+1} \cdot \frac{p(1-p^k)}{1-p}.$$ 
\end{theorem}

\begin{proof}
We show that this formula satisfies the recurrence in Proposition~\ref{prop:newbetti}.  First, observe that the formula yields $0$ whenever $k = 0$ or $k > n$.  Next, suppose that the formula holds for $n-1$ and check that \begin{align*} B_k(n,p) &= B_k(n-1,p) + p \cdot B_{k-1}(n-1,p) + p \cdot \binom{n}{k} \\[1em] &= \binom{n}{k+1} \cdot \frac{p(1-p^k)}{1-p} + p \cdot \binom{n}{k} \cdot \frac{p(1-p^{k-1})}{1-p} + p \cdot \binom{n}{k} \\[1em] &= \binom{n}{k+1} \cdot \frac{p(1-p^k)}{1-p} + \binom{n}{k} \cdot \frac{p(1-p^k)}{1-p} \\[1em] &= \binom{n+1}{k+1} \cdot \frac{p(1-p^k)}{1-p}. \end{align*} This completes the proof.
\end{proof}

The expected Betti sequences for quotients by $2$-linear ideals sampled uniformly at random for $1 \leq n \leq 4$ are listed in Table~\ref{tab:expected-betti}.

\begin{table}
\renewcommand{\arraystretch}{1.2}
\setlength{\tabcolsep}{6pt}
\begin{center}
\begin{tabular}{c||c|c|c|c}
$n$ & $1$ & $2$ & $3$ & $4$ \bigstrut \\ 
\hline
$\mathbb{E}(\beta(T(n,1/2)))$ & $\left(\frac{1}{2}\right)$ & $\left(\frac{3}{2},\frac{3}{4}\right)$ & $\left(3,3,\frac{7}{8}\right)$ & $\left(5,\frac{15}{2},\frac{35}{8},\frac{15}{16}\right)$ \bigstrut \\
\end{tabular}
\end{center}
\caption{The expected Betti sequences with $p = 1/2$ for $1 \leq n \leq 4$.}\label{tab:expected-betti} 
\end{table}

\begin{remark}
The more knowledgeable reader will know that the alternating sum of the Betti numbers is always $0$.  Indeed, the expected alternating sum of the expected Betti numbers from Theorem~\ref{thm:newbetti} is $0$ with the additional observation that the expected value of $\beta_0(\Bbbk[T(n,p)])$ is $1-1/p^n$. 
\end{remark}

\subsection{Random Anti-Lecture Hall Compositions}

Last, but not least, we calculate the expected anti-lecture hall composition associated to $T(n,p)$.  To simplify notation, let $\Lambda_k(n,p) = \mathbb{E}(\lambda_k(T(n,p))$. As we did in the previous section, we begin by establishing a recurrence on $\Lambda_k(n,p)$.  

\begin{proposition}\label{prop:expalhc}
If we set $\Lambda_k(n,p) = 0$ for $k = 0$ and $k > n$, then the triangle of numbers $\Lambda_k(n,p)$ satisfies the recurrence $$\Lambda_k(n,p) = (1-p) \cdot \Lambda_k(n-1,p) + p \cdot \Lambda_{k-1}(n-1,p) + p$$ for $1 \leq k \leq n$.
\end{proposition}

\begin{proof}
This follows from Lemma~\ref{lem:alhc-shift} since every element of $\mathcal{A}_n$ arises from an element of $\mathcal{A}_{n-1}$ via one of the $\delta$ or $\iota$ operations.  We know that elements of $\mathcal{A}_n$ arising from $\delta$ correspond to the same compositions with $0$ appended to their ends and that each will occur with probability $1-p$.  Thus, these contribute $(1-p) \cdot \Lambda_k(n-1,p)$ to $\Lambda_k(n,p)$.  The remaining elements of $\mathcal{A}_n$ arise from applying $\iota$ to each element of $\mathcal{A}_{n-1}$.  This corresponds to shifting each sequence to the right and adding $1$ to each entry and occurs with probability $p$ in each case.  Thus, they contribute $p \cdot (\Lambda_{k-1}(n-1,p) + 1)$ to $\Lambda_k(n,p)$.
\end{proof}

An immediate consequence of Lemma~\ref{lem:uniform} and Proposition~\ref{prop:expalhc} is the surprising result that the expected values of an anti-lecture hall composition bounded above by $1$ are symmetric.

\begin{corollary}
The recurrence in Proposition~\ref{prop:expalhc} becomes $$\Lambda_k(n,1/2) = \frac{1}{2} \big(\Lambda_k(n-1,1/2) +  \Lambda_{k-1}(n-1,1/2) + 1 \big)$$ with $p = 1/2$. \qed  
\end{corollary}

The expected anti-lecture hall compositions bounded above by $1$ sampled uniformly at random for $1 \leq n \leq 4$ are listed in Table~\ref{tab:expalhc}.

\begin{table}
\renewcommand{\arraystretch}{1.2}
\setlength{\tabcolsep}{6pt}
\begin{center}
\begin{tabular}{c||c|c|c|c}
$n$ & 1 & 2 & 3 & 4 \\
\hline 
$\mathbb{E}(\lambda(T(n,1/2)))$ & $\left(\frac{1}{2}\right)$ & $\left( \frac{3}{4}, \frac{3}{4} \right)$ & $\left(\frac{7}{8},\frac{10}{8},\frac{7}{8}\right)$ & $\left( \frac{15}{16}, \frac{25}{16}, \frac{25}{16}, \frac{15}{16} \right)$ \\
\end{tabular}
\medskip
\caption{The expected anti-lecture hall compositions with $p = 1/2$ for $1 \leq n \leq 4$.}\label{tab:expalhc}
\end{center}
\end{table}

Finally, we are able to give a formula for the expected anti-lecture hall composition associated to $T(n,p)$.  

\begin{theorem}\label{thm:expalhc}
For every $1 \leq k \leq n$ and $0 < p < 1$, $$\mathbb{E}(\lambda_k(T(n,p))) = \sum\limits_{i = k}^{n} (-1)^{i+k} \binom{n+1}{i+1} \binom{i-1}{k-1} \frac{p(1-p^i)}{1-p}.$$
\end{theorem}

\begin{proof}
This follows directly from Theorems \ref{thm:newbetti} and \ref{thm:transform}, but it is also straightforward to show that this formula satisfies the recurrence in Proposition~\ref{prop:expalhc}.  
\end{proof}

We can also prove the following formula for the expected values of an anti-lecture hall composition without an alternating sum, although this new formula involves a double sum.  

\begin{theorem}\label{thm:nonnegative-alhc}
\label{prop:formula}The expected entries of the anti-lecture hall composition associated to $T(n,p)$ are given by $$\mathbb{E}\left(\lambda_k\left(n,p\right)\right) = \sum_{1\leq i\leq m\leq j\leq n}\binom{j-i}{m-i} \cdot p^{\ j-m+1} \cdot (1-p)^{m-i}$$ where $m = n-k+1$.  
\end{theorem}

\begin{proof}
We prove that this formula satisfies the recurrence in Proposition~\ref{prop:expalhc}.  Let $$\Lambda_k(n,p) = \sum_{1\leq i\leq m\leq j\leq n}\binom{j-i}{m-i} \cdot p^{j-m+1}q^{m-i}$$ where $m = n-k+1$ and $q = 1-p$.  Then,   
\begin{align*} 
\Lambda_k(n,p) 
&= \sum_{\substack{1\leq i \leq m \\ m \leq j \leq n}}\binom{j-i}{m-i}  p^{j-m+1}  q^{m-i} \\[1em]
&=  \sum_{1\leq i \leq m} p  q^{m-i} + \sum_{\substack{1\leq i < m \\ m < j \leq n}}\binom{j-i}{m-i}  p^{j-m+1}  q^{m-i} + \sum_{m < j \leq n} p^{j-m+1}  \\[1em] 
&= p + \sum_{1\leq i < m} p  q^{m-i} + \sum_{\substack{1\leq i < m \\ m < j \leq n}}\left(\binom{j-i-1}{m-i-1} + \binom{j-i-1}{m-i}\right)  p^{j-m+1}  q^{m-i} + \sum_{m < j \leq n} p^{j-m+1}  \\[1em]
&= p + \sum_{\substack{1\leq i < m \\ m \leq j \leq n}}\binom{j-i-1}{m-i-1} p^{j-m+1}  q^{m-i} + \sum_{\substack{1\leq i \leq m \\ m < j \leq n}}\binom{j-i-1}{m-i}  p^{j-m+1}  q^{m-i} \\[1em]
&= p + \sum_{\substack{1\leq i \leq m-1 \\ m-1 \leq j \leq n-1}}\binom{j-i}{m-i-1} p^{j-m+2}  q^{m-i} + \sum_{\substack{1\leq i \leq m \\ m \leq j \leq n-1}}\binom{j-i}{m-i}  p^{j-m+2}  q^{m-i} \\[1em]
&= p + \Lambda_{k}(n-1,p) + \Lambda_{k-1}(n-1,p).
\end{align*}
This completes the proof.
\end{proof}

\begin{remark}
The formula in Theorem~\ref{thm:nonnegative-alhc} suggests a connection to up/right paths in a lattice rectangle. This raises a natural question about bijections between anti-lecture hall compositions and weighted lattice paths in addition to the many bijections mentioned in \cite{S16}.
\end{remark}

\section{Future Directions}

We hope the results in this paper may contribute to the two larger research programs described below.  

\subsection{Probabilistic Commutative Algebra}

A number of papers exploring probabilistic approaches to commutative algebra have emerged over the past two years.  De Loera, Petrovic, Silverstein, Stasi, and Wilburne \cite{DPSSW} and De Loera, Ho\c{s}ten, Krone, and Silverstein \cite{DHKS} studied the expected homological algebraic properties monomial ideals defined by random generating sets and Erman and Yang \cite{EY18} studied distributions of Betti numbers along fixed rows of Betti tables of Stanley-Reisner ideals of random flag complexes generated using the Erd\H{o}s-R\'enyi random graph model. It would be interesting to investigate the similarities and differences between each of these approaches.  

\subsection{Combinatorics of $d$-Linear Resolutions}

It follows from the work of Herzog, Sharifan, and Varbaro \cite{HSV14} that the Stanley-Reisner ideal of every $(d-1)$-dimensional shifted simplicial complex has a $d$-linear resolution.  Since $1$-dimensional shifted simplicial complexes and threshold graphs are one in the same, it is natural to ask if every Betti table from a  $d$-linear resolution arises from the Stanley-Reisner ideal of a shifted simplicial complex (which is true for $d = 2$ by Theorem~\ref{thm:correspondence}).  While an analogous setup for the proof of Theorem~\ref{thm:correspondence} is plausible for arbitrary dimensions, the combinatorics of shifted simplicial complexes is significantly more complicated when the dimension is greater than one, particularly since those complexes no longer have nice recursive characterizations like the one with dominating and isolated vertices for threshold graphs.  Moreover, one would need to identify an appropriate family of generalized anti-lecture hall compositions corresponding to shifted simplicial complexes, but it is not obvious how to do this.  For more information on generalized anti-lecture hall compositions, see \cite{S16}. Our hope is that the explicit combinatorial mappings in Section~\ref{sec:mappings} may be useful for overcoming these obstacles to generalize Theorem~\ref{thm:correspondence} to higher dimensions.



\bibliographystyle{plain}
\bibliography{references}


\end{document}